\documentclass[12pt]{amsart}

\usepackage[english]{babel}

\usepackage{amsfonts,fullpage}
\usepackage{amsrefs}
\usepackage{amstext}
\usepackage{mathtools}
\usepackage{multicol}
\usepackage{hyperref}

\usepackage{tikz}
\usetikzlibrary{backgrounds,shapes}

\theoremstyle{plain}

\theoremstyle{definition}
\newtheorem{theorem}{Theorem}[section]

\newtheorem{corollary}[theorem]{Corollary}
\newtheorem{definition}[theorem]{Definition}
\newtheorem{conjecture}[theorem]{Conjecture}

\theoremstyle{remark}

\newtheorem{example}[theorem]{Example}

\makeatletter%
\renewenvironment{proof}[1][\proofname]{%
	\par\pushQED{\qed}\normalfont%
	\topsep6\p@\@plus6\p@\relax
	\trivlist\item[\hskip\labelsep\bfseries#1\@addpunct{.}]%
	\ignorespaces
}{%
	\qedhere 
}
\makeatother

\newcommand{\<}{\langle}
\renewcommand{\>}{\rangle}
\newcommand{\qbinom}[2]{\genfrac{[}{]}{0pt}{}{#1}{#2}}

\newcommand{\area}{\mathsf{area}}
\newcommand{\bounce}{\mathsf{bounce}}
\newcommand{\dinv}{\mathsf{dinv}}
\newcommand{\pmaj}{\mathsf{pmaj}}

\DeclareMathOperator{\RP}{RP}

\DeclareMathOperator{\PLDP}{PLDP}

\begin{document}

\title{Parallelogram polyominoes,\\ partially labelled Dyck paths,\\ and the Delta conjecture}

\author[M. D'Adderio]{Michele D'Adderio}
\address{Universit\'e Libre de Bruxelles (ULB)\\D\'epartement de Math\'ematique\\ Boulevard du Triomphe, B-1050 Bruxelles\\ Belgium}\email{mdadderi@ulb.ac.be}

\author[A.Iraci]{Alessandro Iraci}
\address{Universit\'a di Pisa and Universit\'e Libre de Bruxelles (ULB)\\Dipartimento di Matematica\\ Largo Bruno Pontecorvo 5, 56127 Pisa\\ Italia}\email{sashairaci@gmail.com}


\begin{abstract}
	We introduce $\mathsf{area}$, $\mathsf{bounce}$ and $\mathsf{dinv}$ statistics on decorated parallelogram polyominoes, and prove that some of their $q,t$-enumerators match $\langle \Delta_{h_m} e_{n+1}, s_{k+1,1^{n-k}} \rangle$, extending in this way the work in (Aval et al. 2014). Also, we provide a bijective connection between decorated parallelogram polyominoes and decorated labelled Dyck paths, which allows us to prove the combinatorial interpretation of the coefficient $\langle \Delta_{e_{m+n-k-1}}'e_{m+n}, h_m h_n \rangle$ predicted by the Delta conjecture in (Haglund et al. 2015). Finally, we define a statistic $\mathsf{pmaj}$ on partially labelled Dyck paths, which provides another conjectural combinatorial interpretation of $\Delta_{h_{\ell}}\Delta_{e_{n-k-1}}'e_n$, cf. (Haglund et al. 2015).
	
	This is an extended abstract of (D'Adderio, Iraci 2017): this forthcoming publication will have proofs and additional details and results.
\end{abstract}

\maketitle


In \cite{ADDHL}, statistics $\area$, $\bounce$ and $\dinv$ have been defined on parallelogram polyominoes, and some of their $q,t$-enumerators have been shown to be $\langle \Delta_{e_{m+n}}e_{m+n},h_m h_n\rangle = \langle \Delta_{h_m}e_{n+1},s_{1^{n+1}}\rangle$. In particular, the first author proposed a combinatorial intepretation of the full symmetric function $\Delta_{h_m}e_{n+1}$ at $q=1$ in terms of labelled parallelogram polyominoes, and asked for a $\dinv$ statistic that could give the more general $\Delta_{h_m}e_{n+1}$ (cf. \cite[Equations (8.1) and (8.14)]{AvalBergeronGarsia}). 

In \cite{HaglundRemmelWilson} the authors state the so called Delta conjecture, which predicts a combinatorial interpretation of the symmetric function $\Delta_{e_k}e_n$. This is a generalization of the \emph{Shuffle conjecture} stated in \cite{HHLRU} and \cite{HaglundMorseZabrocki} and proved in \cite{CarlssonMellit}, which is related to the famous diagonal harmonics discovered by Garsia and Haiman in their work towards a proof of the Schur positivity of Macdonald polynomials (cf. \cite{GarsiaHaimanPNAS,GarsiaHaimanqLagrange,Haimannfactorial,Haimanvanishing}).

Other than the results mentioned in \cite{HaglundRemmelWilson}, other consequences of this conjecture have been proved, in particular in \cite{DadderioVandenwyngaerd,GarsiaHaglundRemmelYoo,RomeroDeltaq1,Zabrocki4catalan}, while the general Delta conjecture remains open.

The delta operator $\Delta_f$ has been defined for any symmetric function $f$ by Bergeron, Garsia, Haiman and Tesler, and in fact in \cite{HaglundRemmelWilson} the authors provide a generalization of their Delta conjecture for the symmetric function $\Delta_{h_{\ell}e_k}e_n$ in terms of partially labelled Dyck paths.

In this work we extend the results in \cite{ADDHL}, by providing a combinatorial interpretation of the more general $\langle\Delta_{h_m}e_{n+1},s_{k+1,1^{n-k}}\rangle$ in terms of decorated parallelogram polyominoes.

Also, we prove the formula for $\langle \Delta_{e_{m+n-k-1}}'e_{m+n},h_mh_{n} \rangle$ predicted by the Delta conjecture in \cite{HaglundRemmelWilson}, by providing a recursion of these polynomials. This, together with the results in \cite{DadderioVandenwyngaerd}, completely solves Problem 8.1 in \cite{HaglundRemmelWilson}.

Surprisingly, these two solutions are intimately related: indeed, with a bijection, we show that these two combinatorial polynomials actually coincide.

In order to prove these results, we prove some symmetric function identities. The proofs of these are based on theorems in \cite{DadderioVandenwyngaerd} and \cite{HaglundqtSchroeder} 

Finally, we introduce a $\pmaj$ statistic on partially labelled Dyck paths, providing another combinatorial interpretation of $\Delta_{h_{\ell}e_k}e_n$. Then, for the special case $k=0$, we describe a bijection with labelled parallelogram polyominoes, which allows us to define both a $\dinv$ and a $\pmaj$ statistic on these objects: this answers the question in \cite[Equation (8.14)]{AvalBergeronGarsia}.

This is an extended abstract of a forthcoming publication \cite{DadderioIraci}, which will have proofs and additional details and results

\section{Decorated parallelogram polyominoes}

A parallelogram polyomino is a pair of lattice paths from $(0,0)$ to $(m,n)$ such that the first path (the red path) lies \textit{strictly} above the second path (the green path), see Figure \ref{fig:polyomino} for an example. 

In \cite{ADDHL}, authors define three statistics $\area$, $\bounce$, and $\dinv$ on parallelogram polyominoes. The pairs $(\area,\bounce)$ and $(\dinv,\area)$ give rise to a $q,t$-analogue of the Narayana numbers. 

In this article we prefer to work with the almost equivalent \textit{reduced parallelogram polyominoes}. For reduced parallelogram polyominoes, we only ask that the red path lies \textit{weakly} above the green one.

\begin{definition}
	A $m \times n$ \textit{reduced parallelogram polyomino} is a pair of $(m+n)$-tuples $(\textbf{r}, \textbf{g})$, with $\textbf{r} = (r_1,\dots,r_{m+n})$, $\textbf{g} = (g_1,\dots,g_{m+n})$ such that
		
	\begin{itemize}
		\item  $r_i, g_i \in \{0,1\}$,
		\item for all $i$, it holds $g_1 + \dots + g_i \leq r_1 + \dots + r_i$,
		\item $g_1 + \dots + g_{m+n} = r_1 + \dots + r_{m+n} = n$.
	\end{itemize}
So $\textbf{r}$ and $\textbf{g}$ are the red and the green paths, encoded as sequences $1$ and $0$ representing north and east unit steps respectively. See Figure~\ref{fig:polyomino} for an example.
\end{definition}
We now define statistics for suitably decorated reduced polyominoes. While area and bounce will naturally extend the corresponding ones given in \cite{ADDHL}, for the dinv we use a ``reversed'' version, as it suits better our purposes.

\subsection{The area word}

Parallelogram polyominoes can be coded using their \textit{area word} (see \cite{ADDHL}). We do the same for reduced polyominoes, using essentially the same bijection in \cite[Section~3]{ADDHL}: to get the area word, one can interlace the two paths to build a Dyck path $\textbf{D} := (1, r_1, 1 - g_1, r_2, 1 - g_2, \dots, 1 - g_{m+n}, 0)$ of length $2m+2n+2$, which we think of as a sequence of northeast and southeast steps (the $1$s and the $0$s respectively) starting from $(0,0)$ and staying weakly above the $x$ axis. Then we label the positive vertical levels with the ordered alphabet $0 < \bar{0} < 1 < \bar{1} < \dots$ and so on. Starting from $(0,0)$, we go through the path, and whenever we read a northeast step (i.e. a $1$ in $\textbf{D}$), we write down the corresponding level, while we write nothing for the southeast steps. 

It follows from \cite[Section~3]{ADDHL} that this is 
a bijective correspondence between $m \times n$ reduced parallelogram polyominoes and Dyck words of length $m+n+1$ in the ordered alphabet $0 < \bar{0} < 1 < \bar{1} < \dots$ starting with $0$, with exactly $m+1$ unbarred letters, and exactly $n$ barred letters. By a \emph{Dyck word} we mean that any letter is less or equal (in the order of the alphabet) then the successor of the letter to its left.

As suggested by Figure~\ref{fig:polyomino}, the area word can be computed in a different way, described in \cite[Section~2]{ADDHL}. The same procedure works for reduced polyominoes, and gives the same result (except for the starting $0$, which is missing in regular parallelogram polyominoes).

\begin{figure}[!ht]
	\begin{center}
		\begin{tikzpicture}[scale=0.5]
		\draw[step=1.0,gray,opacity=0.6,thin] (0,0) grid (12,7);
		\filldraw[yellow, opacity=0.3] (0,0) -- (3,0) -- (3,1) -- (5,1) -- (5,3) -- (7,3) -- (7,4) -- (10,4) -- (10,5) -- (12,5) -- (12,7) -- (8,7) -- (8,5) -- (5,5) -- (5,4) -- (3,4) -- (3,3) -- (0,3) -- cycle;
		
		\draw[black]
		(1,0) -- (0,1)
		(2,0) -- (0,2)
		(3,0) -- (0,3)
		(4,1) -- (2,3)
		(5,1) -- (3,3)
		(6,3) -- (5,4)
		(7,3) -- (5,5)
		(8,4) -- (7,5)
		(9,4) -- (8,5)
		(10,4) -- (8,6)
		(11,5) -- (9,7)
		(12,5) -- (10,7);
		
		\filldraw[fill=black]
		(2.5,1.5) circle (2pt)
		(1.5,2.5) circle (2pt)
		(4.5,2.5) circle (2pt)
		(3.5,3.5) circle (2pt)
		(4.5,3.5) circle (2pt)
		(6.5,4.5) circle (2pt)
		(9.5,5.5) circle (2pt)
		(8.5,6.5) circle (2pt)
		(11.5,6.5) circle (2pt);
		
		\node[below] at (0.5,0) {$1$};
		\node[below] at (1.5,0) {$2$};
		\node[below] at (2.5,0) {$3$};
		\node[below] at (3.5,1) {$2$};
		\node[below] at (4.5,1) {$2$};
		\node[below] at (5.5,3) {$1$};
		\node[below] at (6.5,3) {$2$};
		\node[below] at (7.5,4) {$1$};
		\node[below] at (8.5,4) {$1$};
		\node[below] at (9.5,4) {$2$};
		\node[below] at (10.5,5) {$2$};
		\node[below] at (11.5,5) {$2$};
		
		\node[left] at (0,0.5) {$\bar{0}$};
		\node[left] at (0,1.5) {$\bar{1}$};
		\node[left] at (0,2.5) {$\bar{2}$};
		\node[left] at (3,3.5) {$\bar{2}$};
		\node[left] at (5,4.5) {$\bar{1}$};
		\node[left] at (8,5.5) {$\bar{1}$};
		\node[left] at (8,6.5) {$\bar{2}$};
		
		\draw[green, line width=3pt] (0,0) -- (3,0) -- (3,1) -- (5,1) -- (5,3) -- (7,3) -- (7,4) -- (10,4) -- (10,5) -- (12,5) -- (12,7);
		\draw[red, line width=3pt] (0,0) -- (0,3) -- (3,3) -- (3,4) -- (5,4) -- (5,5) -- (8,5) -- (8,7) -- (12,7);
		\end{tikzpicture}
	\end{center}
	
	\caption{A $12 \times 7$ polyomino with area word $\bar{0} 1 \bar{1} 2 \bar{2} 3 2 2 \bar{2} 1 \bar{1} 2 1 1 \bar{1} 2 \bar{2} 2 2$.}
	\label{fig:polyomino}
	
\end{figure}
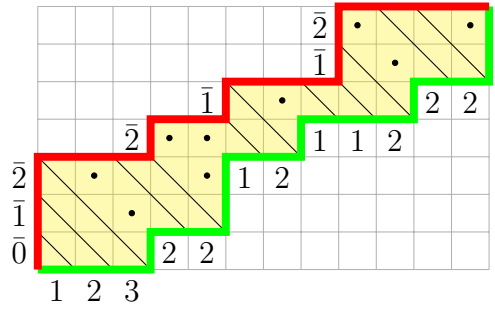

\subsection{The statistics $\area$ and $\underline{\area}$}

On parallelogram polyominoes, one can define a statistic $\area$ as the sum of the letters of the area word. This is actually the area of the polyomino, i.e. the number of unit squares between the two paths.

Let us call \textit{rise} a pair of consecutive letters in the area word of a polyomino such that the former is barred, and the latter is its successor in the alphabet (e.g. $\bar{0}1, \bar{1}2, \bar{2}3, \dots$). Then we can decorate a polyomino by marking the unbarred letters of some of its rises. We define a new statistic $\underline{\area}$ on decorated polyominoes as the sum of the letters of the area word, except the marked ones.

\subsection{The statistics $\bounce$ and $\underline{\bounce}$}

On parallelogram polyominoes we have a second statistic, the $\bounce$. It is computed by drawing the \textit{bounce path}, which is a lattice path going from $(0,0)$ to $(m,n)$, as defined in \cite{ADDHL}. Once again, for reduced polyominoes we need a slight modification of the algorithm. To draw the bounce path, we proceed as follows: starting from $(0,0)$, draw vertical steps until the bounce path hits the beginning of a horizontal red step, then draw horizontal steps until it hits the beginning of a vertical green step, and so on. Continue until you reach $(m,n)$.

Starting from the letter $0$, attach to each step of the bounce path a letter of the usual alphabet, going up a level each time the path changes direction (i.e. when it bounces). Then the $\bounce$ is the sum of the labels attached to the steps of the bounce path. See Figure~\ref{fig:bouncepath} for an example. Notice that there may be no vertical steps labelled with $0$, as both the red and the green path can start horizontally; in that case, we start labelling steps with $\bar{0}$.

Let us call \textit{red valley} a vertical step of the red path preceded by a horizontal step. We can decorate a polyomino by marking some of its valleys, and define a statistic $\underline{\bounce}$ as the sum of the labels attached to the steps of the bounce path, except those lying in the same row as a marked valley (which are unbarred). See Figure~\ref{fig:bouncepath} for an example.

\begin{figure}[!h]
	\begin{center}
		\begin{tikzpicture}[scale=0.6]
			\draw[step=1.0,gray,opacity=0.6,thin] (0,0) grid (12,7);
			\filldraw[yellow, opacity=0.3] (0,0) -- (2,0) -- (2,3) -- (7,3) -- (7,4) -- (10,4) -- (10,5) -- (12,5) -- (12,6) -- (8,6) -- (8,5) -- (5,5) -- (5,4) -- (3,4) -- (3,3) -- (0,3) -- cycle;
			
			\draw[green, line width=3pt] (0,0) -- (2,0) -- (2,3) -- (7,3) -- (7,4) -- (10,4) -- (10,5) -- (12,5) -- (12,7);
			\draw[red, line width=3pt] (0,0) -- (0,3) -- (3,3) -- (3,4) -- (5,4) -- (5,5) -- (8,5) -- (8,6) -- (12,6) -- (12,7);
			
			\filldraw[fill=red]
			(5,4.5) circle (6pt)
			(8,5.5) circle (6pt);
			
			\draw[blue, line width=1.5pt] (0,0) -- (0,3) -- (7,3) -- (7,5) -- (12,5) -- (12,7);
			
			\node[blue, left] at (0,0.5) {$0$};
			\node[blue, left] at (0,1.5) {$0$};
			\node[blue, left] at (0,2.5) {$0$};
			\node[blue, below] at (0.5,3) {$\bar{0}$};
			\node[blue, below] at (1.5,3) {$\bar{0}$};
			\node[blue, below] at (2.5,3) {$\bar{0}$};
			\node[blue, below] at (3.5,3) {$\bar{0}$};
			\node[blue, below] at (4.5,3) {$\bar{0}$};
			\node[blue, below] at (5.5,3) {$\bar{0}$};
			\node[blue, below] at (6.5,3) {$\bar{0}$};
			\node[blue, left] at (7,3.5) {$1$};
			\node[red, left] at (7,4.5) {$1$};
			\node[blue, below] at (7.5,5) {$\bar{1}$};
			\node[blue, below] at (8.5,5) {$\bar{1}$};
			\node[blue, below] at (9.5,5) {$\bar{1}$};
			\node[blue, below] at (10.5,5) {$\bar{1}$};
			\node[blue, below] at (11.5,5) {$\bar{1}$};
			\node[red, left] at (12,5.5) {$2$};
			\node[blue, left] at (12,6.5) {$2$};
		\end{tikzpicture}
	\end{center}
	
	\caption{A polyomino with the bounce path shown, and two marked red valleys. In this case, we should ignore the red letters.}
	\label{fig:bouncepath}
\end{figure}

\subsection{The statistic $\dinv$}

As for regular polyominoes, we have a third statistic, the $\dinv$. Let us call \textit{inversion} any pair of letters in the area word such that the one on the left is the successor, in the usual alphabet, of the one on the right (notice that this is the opposite of what the authors did in \cite{ADDHL}). Then we define the $\dinv$ of a reduced polyomino as the number of its inversions.

We do not give a decorated version for this statistic (though it can be defined), as we do not need it to state our results.

\subsection{A bijection swapping $(\area, \underline{\bounce})$ and $(\dinv, \underline{\area})$}

Let $\RP(m,n)$ be the set of reduced polyominoes of size $m \times n$. Let $\RP(m,n)^{\bullet k}$ be those with $k$ decorated red valleys, and $\RP(m,n)^{\star k}$ those with $k$ decorated rises.

\begin{theorem} \label{thm:zeta}
	There exists a bijection $\zeta \colon \RP(m,n)^{\bullet k} \rightarrow \RP(n,m)^{\star k}$ mapping $(\area, \underline{\bounce})$ to $(\dinv, \underline{\area})$ and swapping valleys and rises.
\end{theorem}

\begin{proof}[Sketch of the proof.]
	The map $\zeta$ that we are going to define is very similar to the one in \cite[Section~4]{ADDHL}.
	Given a polyomino in $\RP(m,n)^{\bullet k}$, we want to write down the area word of a polyomino in $\RP(n,m)^{\star k}$ using the labels of its bounce path.
	
	We start writing down an amount of $0$'s equal to the number of $0$'s appearing as labels in the bounce path, plus one (we need to artificially add the initial $0$). Then we ``project'' the $0$ labels and the $\bar{0}$ labels of the bounce path on the steps of the green path to the right and below them, respectively, and we read them following the green path from $(0,0)$, bottom to top, left to right. We insert now the $\bar{0}$'s among our $0$'s according to the interlacing order we just read (of course, ignoring the first artificial $0$).
	
	After that, we repeat the procedure projecting $\bar{0}$'s and $1$'s on the steps of the red path going from the first to the second bounce point on the red path, and so on. In the image, we decorate a rise if and only if the corresponding red valley is decorated. We iterate these steps until we reach the end of the bounce path.
	
	This $\zeta$ obviously maps $\underline{\bounce}$ to $\underline{\area}$, since we are just taking an anagram of the labels while preserving decorations. It's easy to check that $\zeta$ also maps $\area$ to $\dinv$, since the inversions in the image correspond to the squares inside the polyomino. To go back, reorder in a weakly increasing way the area word to get the bounce path, and then draw the red and green a paths between the bouncing points according to the interlace in the original area word among consecutive (in the alphabet) letters.
\end{proof}

\begin{example}
	Take the polyomino in Figure~\ref{fig:bouncepath}. The interlacing between $0$'s and $\bar{0}$'s is $\bar{0} \bar{0} 0 0 0 \bar{0} \bar{0} \bar{0} \bar{0} \bar{0}$, the interlacing between $\bar{0}$'s and $1$'s gives $\bar{0} \bar{0} \bar{0} 1 \bar{0} \bar{0} {\color{red} 1} \bar{0} \bar{0}$, and so on. Notice that the decorated ${\color{red} 1}$ has been mapped to a rise. The image by $\zeta$ of the polyomino in Figure~\ref{fig:bouncepath} has area word $0 \bar{0} \bar{0} 0 0 0 \bar{0} 1 \bar{1} {\color{red} 2} \bar{1} \bar{1} \bar{0} \bar{0} {\color{red} 1} \bar{1} \bar{1} 2 \bar{0} \bar{0}$.
\end{example}

\subsection{Recursions for $(\area, \underline{\bounce})$ and $(\dinv, \underline{\area})$}

Let $\RP(m,n \backslash s)^{\bullet k}$ be the set of $m \times n$ reduced polyominoes with $k$ marked valleys and $s-1$ labels with value $0$ in the bounce path. Let \[ \RP_{q,t}(m,n \backslash s)^{\bullet k} \coloneqq \sum_{P \in \RP(m,n \backslash s)^{\bullet k}} q^{\area(P)} t^{\underline{\bounce}(P)}. \]

\begin{theorem}
	\label{th:bouncerecursion}
	For $m \geq 1$, $n \geq 1$, $k \geq 0$, and $1 \leq s \leq n+1$, the polynomials $\RP_{q,t}(m,n \backslash s)^{\bullet k}$ satisfy the recursion
	
	\begin{align*}
		\RP_{q,t}(m,n \backslash s)^{\bullet k} & = \sum_{r=1}^{m} t^{m+n+1-r-s-k} \qbinom{r+s-1}{r}_q \sum_{h=0}^{k} q^{\binom{h}{2}} \qbinom{r}{h}_q \\ & \quad \times \sum_{v=1}^{n-s+1} \qbinom{r+v-h-1}{r-1}_q \RP_{q,t}(m-r, n-s \backslash v)^{\bullet k-h}
	\end{align*}
	
	with initial conditions \[ \RP_{q,t}(m, n \backslash n+1)^{\bullet k} = \qbinom{m+n}{m}_q \] and $\RP_{q,t}(0, n \backslash s)^{\bullet k} = 1$ if $k=0$ and $s=n+1$, $0$ otherwise.
\end{theorem}

\begin{proof}[Sketch of the proof.]
	Given a polyomino in $\RP(m,n \backslash s)^{\bullet k}$, whose bounce path has $r$ $\bar{0}$'s and $v$ $1$'s ($h$ of which are decorated), we can cut it one step after the second bounce and get a polyomino in $\RP(m-r, n-s \backslash v)^{\bullet k-h}$: see Figure~\ref{fig:bouncerecursion}.
	
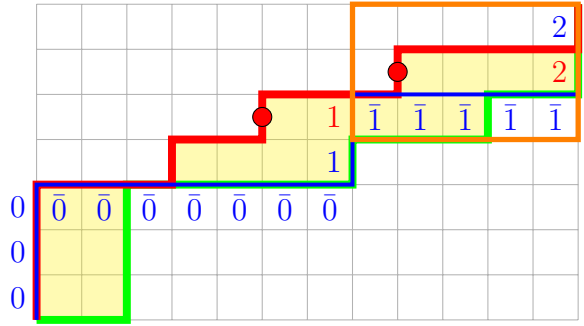
	\begin{figure}[!h]
		\begin{center}
			\begin{tikzpicture}[scale=0.6]
			\draw[step=1.0,gray,opacity=0.6,thin] (0,0) grid (12,7);
			\filldraw[yellow, opacity=0.3] (0,0) -- (2,0) -- (2,3) -- (7,3) -- (7,4) -- (10,4) -- (10,5) -- (12,5) -- (12,6) -- (8,6) -- (8,5) -- (5,5) -- (5,4) -- (3,4) -- (3,3) -- (0,3) -- cycle;
			
			\draw[green, line width=3pt] (0,0) -- (2,0) -- (2,3) -- (7,3) -- (7,4) -- (10,4) -- (10,5) -- (12,5) -- (12,7);
			\draw[red, line width=3pt] (0,0) -- (0,3) -- (3,3) -- (3,4) -- (5,4) -- (5,5) -- (8,5) -- (8,6) -- (12,6) -- (12,7);
			
			\filldraw[fill=red]
			(5,4.5) circle (6pt)
			(8,5.5) circle (6pt);
			
			\draw[blue, line width=1.5pt] (0,0) -- (0,3) -- (7,3) -- (7,5) -- (12,5) -- (12,7);
			\draw[orange, line width=2pt] (7,4) rectangle (12,7);
			
			\node[blue, left] at (0,0.5) {$0$};
			\node[blue, left] at (0,1.5) {$0$};
			\node[blue, left] at (0,2.5) {$0$};
			\node[blue, below] at (0.5,3) {$\bar{0}$};
			\node[blue, below] at (1.5,3) {$\bar{0}$};
			\node[blue, below] at (2.5,3) {$\bar{0}$};
			\node[blue, below] at (3.5,3) {$\bar{0}$};
			\node[blue, below] at (4.5,3) {$\bar{0}$};
			\node[blue, below] at (5.5,3) {$\bar{0}$};
			\node[blue, below] at (6.5,3) {$\bar{0}$};
			\node[blue, left] at (7,3.5) {$1$};
			\node[red, left] at (7,4.5) {$1$};
			\node[blue, below] at (7.5,5) {$\bar{1}$};
			\node[blue, below] at (8.5,5) {$\bar{1}$};
			\node[blue, below] at (9.5,5) {$\bar{1}$};
			\node[blue, below] at (10.5,5) {$\bar{1}$};
			\node[blue, below] at (11.5,5) {$\bar{1}$};
			\node[red, left] at (12,5.5) {$2$};
			\node[blue, left] at (12,6.5) {$2$};
			\end{tikzpicture}
		\end{center}
		
		\caption{After one step of the recursion, we get the polyomino in the orange box.}
		\label{fig:bouncerecursion}
	\end{figure}

	The bounce of the original polyomino can be obtained by adding $m+n+1-r-s-k$ to the bounce of the smaller one. In fact, we are increasing by $1$ all the $m-r+n-s$ labels of the smaller polyomino, adding some $0$'s, $\bar{0}$'s (that do not contribute), and a single $1$ before those labels (which gives the $+1$), and then ignoring the $k$ marked letters ($h$ $0$'s that have become $1$'s, and $k-h$ other letters that increased by one, but have to be ignored).
	
	The area changes as follows. The green path from $(0,0)$ to $(r,s-1)$ gives a contribution taken into account by the first $q$-binomial (notice that the step from $(r,s-1)$ to $(r,s)$ must be vertical). The red path from $(0,s-1)$ to $(r, s+v-1)$ gives a contribution that can be split as follows: at first we ignore the decorated valleys, drawing a path from $(0,s-1)$ to $(r,s+v-h-1)$ which gives a contribution taken into account by the third $q$-binomial; then we add the $h$ decorated valleys that occur before the second bounce, that have to be placed in $h$ different spots (because they are valleys) among the $r$ available ones, for a contribution given by the second $q$-binomial together with its corrective power of $q$.
	
	Putting these together gives us the desired recursion.
\end{proof}

Let $\RP(n \backslash s,m)^{\star k}$ be the set of $n \times m$ reduced polyominoes with $k$ decorated rises and $s$ many $0$'s in the area word, and let $\RP_{q,t}(n \backslash s,m)^{\star k}$ be its $(\dinv, \underline{\area})$ $q,t$-enumerator.

\begin{corollary}
The recursion in Theorem~\ref{th:bouncerecursion} holds also for $\RP_{q,t}(n \backslash s,m)^{\star k}$.
\end{corollary}
This can be proved either directly as we did for $(\area, \underline{\bounce})$, or using the $\zeta$ map of Theorem~\ref{thm:zeta}.

\section{Two cars parking functions}

A \textit{parking function} is a Dyck path with positive integer labels attached to vertical steps such that they are strictly increasing along columns (bottom to top). A \textit{two cars parking function} is a parking function where the labels can only have value $1$ or $2$.

In his PhD thesis \cite{WilsonPhD}, Wilson stated some conjectures about these objects; we prove one of those conjectures, giving also another interpretation of the result in terms of reduced polyominoes.

\begin{figure}[!ht]
	\begin{center}
		\begin{tikzpicture}[scale=0.5]
			\draw[step=1.0, gray, opacity=0.6, thin] (0,0) grid (11,11);
			\draw[gray, opacity=0.6, thin] (0,0) -- (11,11);
			
			\draw[red, line width=2pt] (0,0) -- (0,1) -- (1,1) -- (1,2) -- (1,3) -- (2,3) -- (2,4) -- (3,4) -- (3,5) -- (4,5) -- (4,6) -- (5,6) -- (5,7) -- (6,7) -- (6,8) -- (7,8) -- (8,8) -- (8,9) -- (9,9) -- (9,10) -- (10,10) -- (10,11) -- (11,11);
			
			\node[red] at (0.5,0.5) {$1$};
			\node[red] at (1.5,1.5) {$1$};
			\node[red] at (1.5,2.5) {$2$};
			\node[red] at (2.5,3.5) {$2$};
			\node[red] at (3.5,4.5) {$1$};
			\node[red] at (4.5,5.5) {$1$};
			\node[red] at (5.5,6.5) {$1$};
			\node[red] at (6.5,7.5) {$2$};
			\node[red] at (8.5,8.5) {$1$};
			\node[red] at (9.5,9.5) {$2$};
			\node[red] at (10.5,10.5) {$1$};
		\end{tikzpicture}
	\end{center}
	
	\caption{A two cars parking function with seven $1$'s and four $2$'s.}
	\label{fig:2cpf}
	
\end{figure}
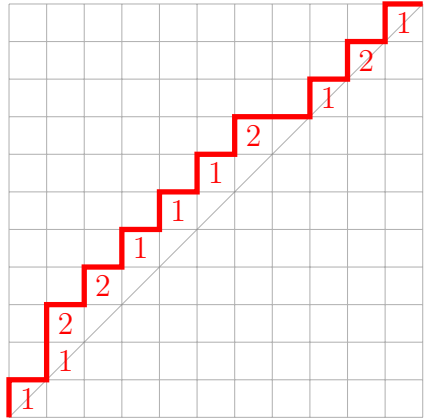

\begin{theorem}
	\label{th:twocarparkfunc}
	There is a bijection between $\RP(m,n)^{\star k}$ and the set of two cars parking functions with $n$ $1$'s, $m$ $2$'s, and $k$ decorated rises. It preserves the bistatistic $(\dinv, \underline{\area})$.
\end{theorem}

\begin{proof}[Scketch of the proof]
	The bijection is extremely simple (and it motivates our definition of $\dinv$): given a parking function, take its area word and put a bar on the letters labelled by a $1$, then add a $0$ at the beginning. This way you get the area word of a reduced polyomino, and both $\dinv$ and $\underline{\area}$ are preserved.
\end{proof}

Let $\mathcal{T}_{q,t}(n\backslash s,m)^{\star k}$ be the polynomial $\sum_{P}q^{\dinv(P)} t^{\underline{\area}(P)}$, where the sum is over two cars parking functions with $m$ $1$'s, $n$ $2$'s of which $s-1$ on the main diagonal, and $k$ decorated rises.

The following corollary of Theorem~\ref{th:twocarparkfunc} is essentially \cite{WilsonPhD}*{Proposition~5.3.3.1}.
\begin{corollary}[Wilson] \label{cor:Wilson}
The polynomial $\mathcal{T}_{q,t}(n\backslash s,m)^{\star k}$ 
satisfies the same recursion as $\RP_{q,t}(m,n\backslash s)^{\bullet k}$ in Theorem~\ref{th:bouncerecursion}. 
\end{corollary}
In particular, $\mathcal{T}_{q,t}(n\backslash s,m)^{\star k}=\RP_{q,t}(m,n\backslash s)^{\bullet k}$. We use this connection to prove \cite{WilsonPhD}*{Conjecture~5.2.2.1}.

%

\section{Symmetric functions and combinatorial interpretations}

All those combinatorial polynomials have an interpretation in terms of symmetric functions, that we are going to state. First of all, we establish the following symmetric function identities (see for example \cite{DadderioVandenwyngaerd} for the undefined notations).

\begin{theorem}
	Let $m \geq 0$, $n \geq 0$ and $0 \leq k \leq n$. Then
	\[ \< \Delta_{h_m} e_{n+1}, s_{k+1,1^{n-k}} \> = \< \Delta_{e_{m+n-k-1}}' e_{m+n}, h_m h_n \>. \]
\end{theorem}

\begin{theorem}
	Let $m \geq 0$, $n \geq 0$ and $0 \leq k \leq n$. Then
	\[ \sum_{r=1}^{m-k} t^{m-r-k} \< \Delta_{h_{m-r-k}} \Delta_{e_k} e_n \left[X \frac{1-q^r}{1-q} \right], e_n \> = \< \Delta_{h_m} e_{n+1}, s_{k+1,1^{n-k}} \>. \]
\end{theorem}

Our main result is the following theorem.

\begin{theorem} \label{thm:lemma}
\begin{equation} \label{eq:thm_lemma}
 \RP_{q,t}(m \backslash r, n)^{\star k} = t^{m-r-k} \< \Delta_{h_{m-r-k}} \Delta_{e_k} e_n \left[ X \dfrac{1 - q^r}{1 - q} \right], e_n \>. 
\end{equation}
\end{theorem}

\begin{proof}[Sketch of the proof.]
	The right hand side of \eqref{eq:thm_lemma} can be written in terms of the $F_{n,k}^{(d,\ell)}$ of \cite[Section~4]{DadderioVandenwyngaerd} using \cite[Equation (1.46)]{DadderioVandenwyngaerd}. Now this and \cite{DadderioVandenwyngaerd}*{Theorem~4.6} give a recursion for our polynomials. A suitable manipulation (involving a change of variables) can bring it to the form stated in Theorem~\ref{th:bouncerecursion}.
\end{proof}

All these results together give the following theorem.

\begin{theorem}
	\[ \RP_{q,t}(m,n)^{\star k} = \< \Delta_{h_m} e_{n+1}, s_{k+1,1^{n-k}} \> = \< \Delta_{e_{m+n-k-1}}' e_{m+n}, h_m h_n \> \]
\end{theorem}

Combining this with Corollary~\ref{cor:Wilson}, this results proves the case $h_m h_n$ of the Delta conjecture in \cite{HaglundRemmelWilson}, i.e. \cite{WilsonPhD}*{Conjecture~5.2.2.1}.

In fact, the recursion mentioned in the proof of Theorem~\ref{thm:lemma}, together with the results in \cite{DadderioVandenwyngaerd}, provides a complete solution of Problem~8.1 in \cite{HaglundRemmelWilson}.

\section{Partially labelled Dyck paths}

Other objects involved in this theory are the partially labelled Dyck paths defined in \cite{HaglundRemmelWilson}*{Section 7.2}. Let us consider the partially labelled Dyck paths of size $m+n+1$ with $n+1$ labels, $m$ blanks (that must be valleys), and $k$ decorated rises. Let $\PLDP(m,n)^{\star k}$ be this set. We are especially interested in the case $k = n$, which means that all the rises are decorated and all the valleys are blanks.

In \cite{HaglundRemmelWilson}, the authors define two statistics $\dinv$ and $\underline{\area}$ for these objects. We can actually define a third statistic $\pmaj$ essentially as the usual one defined in \cite{LoehrRemmelpmaj,Loehrpmaj}. We proceed as follows. First of all, add a $0$ label to every unlabelled vertical step. Then, we pursue this algorithm. At step $0$, we start with an empty multiset $S$. Next, at step $i$, we add to $S$ all the labels appearing in the column $i$ (if any), and write down the greatest letter in $S$ lesser or equal than the one written at step $i-1$ (if any), or the greatest among all the letters in $S$ if we can't; now we remove from $S$ the letter we just wrote, and go on until we reach the last column. Finally, we reverse the word we just wrote and define the $\pmaj$ as the major index of this word, i.e. the sum of the positions of the descents in this reversed word.

\begin{conjecture}
	\label{cong:haglund}
	\[ \Delta_{h_m} \Delta'_{e_{n-k}} e_{n+1} = \sum_{P \in \PLDP(m,n)^{\star k}} q^{\underline{\area}(P)} t^{\pmaj(P)} x^P, \]
	where $x^P$ is the product of the $x_i$ as $i$ runs over the nonzero labels of $P$.
\end{conjecture}

Other than computer verification, in support of our conjecture, in the case $m=0$, i.e. in the Delta conjecture case, the predicted scalar product with $e_jh_{n-j+1}$ is exactly the combinatorial interpretation with the \emph{second bounce} (denoted $\bounce '$) of \cite[Theorem~6.2]{DadderioVandenwyngaerd}.

For $k = n$, there is a bijection from these objects to \emph{labelled reduced polyominoes}, where the labelling (with positive integers) is done by putting a label below the first column, then the other labels on vertical red steps such that columns are strictly increasing (i.e. we are adding an artificial vertical red step at the beginning).


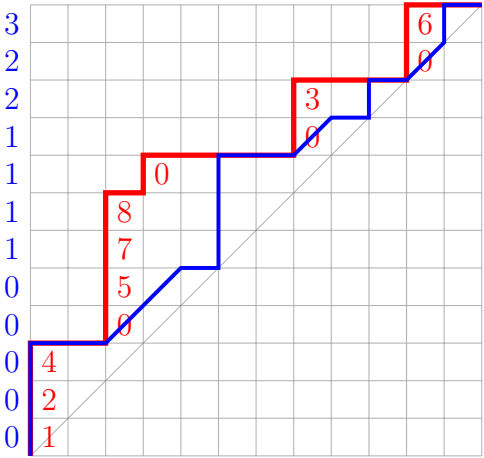
\begin{figure}[!h]
	\begin{center}
		\begin{tikzpicture}[scale=0.5]
		\draw[step=1.0, gray, opacity=0.6, thin] (0,0) grid (12,12);
		
		\draw[gray, opacity=0.6, thin] (0,0) -- (12,12);
		
		\draw[red, line width=2pt] (0,0) -- (0,1) -- (0,2) -- (0,3) -- (1,3) -- (2,3) -- (2,4) -- (2,5) -- (2,6) -- (2,7) -- (3,7) -- (3,8) -- (4,8) -- (5,8) -- (6,8) -- (7,8) -- (7,9) -- (7,10) -- (8,10) -- (9,10) -- (10,10) -- (10,11) -- (10,12) -- (11,12) -- (12,12);
		
		\node[red] at (0.5,0.5) {$1$};
		\node[red] at (0.5,1.5) {$2$};
		\node[red] at (0.5,2.5) {$4$};
		\node[red] at (2.5,3.5) {$0$};
		\node[red] at (2.5,4.5) {$5$};
		\node[red] at (2.5,5.5) {$7$};
		\node[red] at (2.5,6.5) {$8$};
		\node[red] at (3.5,7.5) {$0$};
		\node[red] at (7.5,8.5) {$0$};
		\node[red] at (7.5,9.5) {$3$};
		\node[red] at (10.5,10.5) {$0$};
		\node[red] at (10.5,11.5) {$6$};
		
		\draw[blue, line width=1.5pt] (0,0) -- (0,3) -- (1,3) -- (2,3) -- (3,4) -- (4,5) -- (5,5) -- (5,8) -- (6,8) -- (7,8) -- (8,9) -- (9,9) -- (9,10) -- (10,10) -- (11,11) -- (11,12) -- (12,12);
		
		\node[blue, left] at (0.0,0.5) {$0$};
		\node[blue, left] at (0.0,1.5) {$0$};
		\node[blue, left] at (0.0,2.5) {$0$};
		\node[blue, left] at (0.0,3.5) {$0$};
		\node[blue, left] at (0.0,4.5) {$0$};
		\node[blue, left] at (0.0,5.5) {$1$};
		\node[blue, left] at (0.0,6.5) {$1$};
		\node[blue, left] at (0.0,7.5) {$1$};
		\node[blue, left] at (0.0,8.5) {$1$};
		\node[blue, left] at (0.0,9.5) {$2$};
		\node[blue, left] at (0.0,10.5) {$2$};
		\node[blue, left] at (0.0,11.5) {$3$};
		\end{tikzpicture}
	\end{center}
	
	\caption{A partially labelled Dyck path with $n=7$ and $m=4$. The bounce path is shown (its labels are blue): it goes diagonally in columns containing a $0$ label.}
	\label{fig:pldp}
\end{figure}

\begin{theorem}
	There is a bijection between $\PLDP(m,n)^{\star n}$ and labelled reduced polyominoes of size $m \times n$.
\end{theorem}

\begin{proof}
	Let $P \in \PLDP(m,n)^{\star n}$. Write its area word (the usual one for Dyck paths), colouring in green the numbers corresponding to rows containing a valley, and in red the other ones (for example, the Dyck path in Figure~\ref{fig:pldp} has area word ${\color{red} 0} {\color{red} 1} {\color{red} 2} {\color{green} 1} {\color{red} 2} {\color{red} 3} {\color{red} 4} {\color{green} 4} {\color{green} 1} {\color{red} 2} {\color{green} 0} {\color{red} 1}$).
	
	Now, draw the red path of the polyomino as follows: starting from the second letter of the area word, draw a vertical step if the letter is red, and a horizontal step if it is green. Whenever you draw a vertical step, attach the label in the corresponding row (there must be one, since vertical steps correspond to rows that do not contain valleys).
	
	Next, draw the green path as follows. Draw a horizontal green step $x$ rows below each horizontal red step, where $x$ is the value of the green letter corresponding to that horizontal red step. Then, connect them with vertical steps to get a lattice path from $(0,0)$ to $(m,n)$. See Figure~\ref{fig:bijection} for an example.
%
\end{proof}

\begin{figure}[!h]
	\begin{center}
		\begin{tikzpicture}[scale=0.5]
			\draw[step=1.0, gray, opacity=0.6,thin] (0,0) grid (4,7);
			
			\filldraw[yellow, opacity=0.3] (0,0) -- (0,1) -- (1,1) -- (2,1) -- (2,2) -- (2,3) -- (2,4) -- (3,4) -- (3,5) -- (3,6) -- (4,6) -- (4,7) -- (4,6) -- (3,6) -- (3,5) -- (2,5) --  (1,5) -- (1,4) -- (1,3) -- (1,2) -- (0,2) -- (0,1) -- (0,0);
			
			\draw[green, line width=2.5pt] (0,0) -- (0,1) -- (1,1) -- (2,1) -- (2,2) -- (2,3) -- (2,4) -- (3,4) -- (3,5) -- (3,6) -- (4,6) -- (4,7);
			
			\draw[red, line width=2.5pt] (0,0) -- (0,1) -- (0,2) -- (1,2) -- (1,3) -- (1,4) -- (1,5) -- (2,5) -- (3,5) -- (3,6) -- (4,6) -- (4,7);
					
			\node[red, left] at (0.0,-0.5) {$0$};
			\node[red, left] at (0.0,0.5) {$1$};
			\node[red, left] at (0.0,1.5) {$2$};
			\node[green, left] at (1.0,1.5) {$1$};
			\node[red, left] at (1.0,2.5) {$2$};
			\node[red, left] at (1.0,3.5) {$3$};
			\node[red, left] at (1.0,4.5) {$4$};
			\node[green, left] at (2.0,4.5) {$4$};
			\node[green, left] at (3.0,4.5) {$1$};
			\node[red, left] at (3.0,5.5) {$2$};
			\node[green, left] at (4.0,5.5) {$0$};
			\node[red, left] at (4.0,6.5) {$1$};
		\end{tikzpicture}
		\begin{tikzpicture}[scale=0.5]
			\draw[step=1.0, gray, opacity=0.6,thin] (0,0) grid (4,7);
			
			\filldraw[yellow, opacity=0.3] (0,0) -- (0,1) -- (1,1) -- (2,1) -- (2,2) -- (2,3) -- (2,4) -- (3,4) -- (3,5) -- (3,6) -- (4,6) -- (4,7) -- (4,6) -- (3,6) -- (3,5) -- (2,5) --  (1,5) -- (1,4) -- (1,3) -- (1,2) -- (0,2) -- (0,1) -- (0,0);
			
			\draw[green, line width=2.5pt] (0,0) -- (0,1) -- (1,1) -- (2,1) -- (2,2) -- (2,3) -- (2,4) -- (3,4) -- (3,5) -- (3,6) -- (4,6) -- (4,7);
			
			\draw[red, line width=2.5pt] (0,0) -- (0,1) -- (0,2) -- (1,2) -- (1,3) -- (1,4) -- (1,5) -- (2,5) -- (3,5) -- (3,6) -- (4,6) -- (4,7);
						
			\draw[blue, line width=1.5pt] (0,0) -- (0,1) -- (0,2) -- (1,2) -- (2,2) -- (2,3) -- (2,4) -- (2,5) -- (3,5) -- (3,6) -- (4,6) -- (4,7);
			
			\node[white, left] at (-4.0,-0.5) {$0$};
			\node[blue, right] at (0.0,0.5) {$0$};
			\node[blue, right] at (0.0,1.5) {$0$};
			\node[blue, above] at (0.5,2.0) {$\bar{0}$};
			\node[blue, above] at (1.5,2.0) {$\bar{0}$};
			\node[blue, right] at (2.0,2.5) {$1$};
			\node[blue, right] at (2.0,3.5) {$1$};
			\node[blue, right] at (2.0,4.5) {$1$};
			\node[blue, above] at (2.5,5.0) {$\bar{1}$};
			\node[blue, right] at (3.0,5.5) {$2$};
			\node[blue, above] at (3.5,6.0) {$\bar{2}$};
			\node[blue, right] at (4.0,6.5) {$3$};
		\end{tikzpicture}
	\end{center}

	\caption{The image of the partially labelled Dyck path in Figure~\ref{fig:pldp} with the bijection. The statistics $\area$ and $\bounce$ are preserved.}
	\label{fig:bijection}
\end{figure}
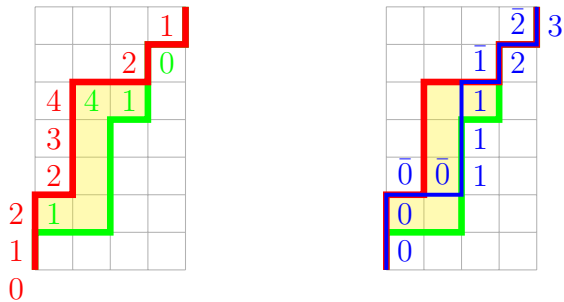

The bijection we just defined maps $\underline{\area}$ to the area of the polyomino (by construction). It also defines a $\pmaj$ statistic on labelled polyominoes (which is easy to read on the image) that collapses to the $\bounce$ statistic if the labels are ordered from $1$ to $n+1$ bottom to top; this gives another combinatorial interpretation for the Conjecture~\ref{cong:haglund} in the case $k=n$. 


Finally, it defines a new $\dinv$ statistic on polyominoes, answering the question in \cite[Equation (8.14)]{AvalBergeronGarsia}.

If viewed in the other direction, this map gives a way to define a $\bounce$ statistic on partially labelled Dyck paths, together with a way to draw a bounce path. An example is shown in Figure~\ref{fig:pldp}.

\bibliographystyle{plain}
\bibliography{Bibliography}

\end{document}